\documentclass[a4paper,10pt,english]{smfart}
\usepackage[english]{babel}
\usepackage[T1]{fontenc}
\usepackage{lmodern}
\usepackage{graphicx}
\usepackage{tabularx}
\usepackage{amsmath,amsthm,amssymb,amsfonts}
\usepackage[a4paper,left=4cm,right=4cm,top=4cm,bottom=4cm]{geometry}
\numberwithin{equation}{section}
\title[Linear growth for certain elliptic fibrations]
{Integral points on quadratic twists and linear growth for certain elliptic fibrations}
\author{Pierre Le Boudec}
\subjclass{$11$D$45$, $11$G$05$, $14$G$05$}
\address{Institute for Advanced Study \\ School of Mathematics \\ Einstein Drive \\
Simonyi Hall $-$ \text{Office $111$} \\ Princeton, NJ $08540$ \\ USA}
\email{pleboudec@ias.edu}

\begin{document}

\makeatletter
\def\imod#1{\allowbreak\mkern10mu({\operator@font mod}\,\,#1)}
\makeatother

\newtheorem{lemma}{Lemma}
\newtheorem{theorem}{Theorem}
\newtheorem{corollary}{Corollary}
\newtheorem{proposition}{Proposition}
\newtheorem{conjecture}{Conjecture}

\newcommand{\vol}{\operatorname{vol}}
\newcommand{\D}{\mathrm{d}}
\newcommand{\rank}{\operatorname{rank}}
\newcommand{\Pic}{\operatorname{Pic}}
\newcommand{\Gal}{\operatorname{Gal}}
\newcommand{\meas}{\operatorname{meas}}
\newcommand{\Spec}{\operatorname{Spec}}
\newcommand{\eff}{\operatorname{eff}}
\newcommand{\rad}{\operatorname{rad}}
\newcommand{\sq}{\operatorname{sq}}
\newcommand{\sign}{\operatorname{sign}}

\begin{abstract}
We prove that the number of rational points of bounded height on certain del Pezzo surfaces of degree $1$ defined over $\mathbb{Q}$
grows linearly, as predicted by Manin's conjecture. Along the way, we investigate the average number of integral points of small
naive height on quadratic twists of a fixed elliptic curve with full rational $2$-torsion.
\end{abstract}

\maketitle

\tableofcontents

\section{Introduction}

\subsection{Rational points on elliptic fibrations}

The main goal of this article is to establish sharp bounds for the number of rational points of bounded height on certain
del Pezzo surfaces of degree $1$ defined over $\mathbb{Q}$. In their anticanonical embedding, these surfaces are defined by sextic
forms in $\mathbb{P}(3,2,1,1)$. More precisely, they are isomorphic to a surface $V$ given by an equation of the shape
\begin{equation}
\label{equation dP1}
y^2 = x^3 + F_4(u,v)x + F_6(u,v),
\end{equation}
where the coordinates in $\mathbb{P}(3,2,1,1)$ are denoted by $(y:x:u:v)$ to highlight the elliptic fibration and where
$F_4, F_6 \in \mathbb{Z}[u,v]$ are respectively a quartic and a sextic form such that $4 F_4^3 + 27 F_6^2$ is not
identically $0$.

For $\mathbf{x} = (y:x:u:v) \in \mathbb{P}(3,2,1,1)(\mathbb{Q})$, we can choose coordinates $y,x,u,v \in \mathbb{Z}$ such that for
every prime $p$, either $p \nmid u$ or $p \nmid v$ or $p^2 \nmid x$ or $p^3 \nmid y$. Then we can define an exponential height
function $H : \mathbb{P}(3,2,1,1)(\mathbb{Q}) \to \mathbb{R}_{>0}$ by setting
\begin{equation*}
H(\mathbf{x}) = \max \{ |y|^{1/3}, |x|^{1/2}, |u|, |v| \}.
\end{equation*}
For any Zariski open subset $U$ of $V$, we can introduce the number of rational points of bounded height on $U$, that is
\begin{equation*}
N_{U,H}(B) = \# \{ \mathbf{x} \in U(\mathbb{Q}), H(\mathbf{x}) \leq B \}.
\end{equation*}
A conjecture of Manin (see \cite{MR974910}) predicts the asymptotic behaviour of $N_{V,H}(B)$ as $B$ tends to $+ \infty$,
but the current technology is very far from allowing us to approach it for any surface $V$. A weaker version states that $V$ has
linear growth, by which we mean that there should exist an open subset $U$ of $V$ such that, for any fixed $\varepsilon > 0$,
\begin{equation}
\label{Manin}
N_{U,H}(B) \ll B^{1 + \varepsilon}.
\end{equation}

The only authors who have addressed this problem seem to be Munshi (see \cite{MR2326487} and \cite{MR2429642}) and
Mendes da Costa (see \cite{MdC}).

More precisely, Mendes da Costa established that for any surface $V$ given by an equation
of the shape \eqref{equation dP1}, there exists $\delta > 0$ such that $N_{V,H}(B) \ll B^{3 - \delta}$, where the constant involved
in the notation $\ll$ is independent of the forms $F_4$ and $F_6$. This bound is far from the expectation
\eqref{Manin} but is not at all trivial, which illustrates the difficulty of this problem in general.

As already remarked by Munshi, it is easier to deal with certain specific examples of singular surfaces. The most striking
result in Munshi's works is the following (see \cite[Corollary~$3$]{MR2429642}). Let
$V_{\mathbf{e},\lambda,R} \subset \mathbb{P}(3,2,1,1)$ be the surface defined by
\begin{equation}
\label{equation Munshi}
y^2 = (x - e R(u,v)) (x - \lambda R(u,v)) (x - \overline{\lambda} R(u,v)),
\end{equation}
where $e \in  \mathbb{Z}$, $\lambda$ is a generator of the ring of integers of an imaginary quadratic field, and
$R \in \mathbb{Z}[u,v]$ is a positive definite quadratic form. Then we have
\begin{equation}
\label{bound Munshi}
N_{U_{\mathbf{e},\lambda,R},H}(B) \ll B^{5/4 + \varepsilon},
\end{equation}
where $U_{\mathbf{e},\lambda,R}$ is defined by removing from $V_{\mathbf{e},\lambda,R}$ the subset defined by $y = 0$.
Although impressive, this result is still far from the conjectured upper bound \eqref{Manin}.

Let $e_1, e_2, e_3 \in  \mathbb{Z}$ be three distinct integers and set $\mathbf{e} = (e_1,e_2,e_3)$. We also let
$Q \in \mathbb{Z}[u,v]$ be a non-degenerate quadratic form. In this article, we are interested in the surfaces
$V_{\mathbf{e},Q} \subset \mathbb{P}(3,2,1,1)$ defined by
\begin{equation}
\label{equation surface}
y^2 = (x - e_1 Q(u,v)) (x - e_2 Q(u,v)) (x - e_3 Q(u,v)).
\end{equation}
We let $U_{\mathbf{e},Q}$ be the open subset defined by removing from $V_{\mathbf{e},Q}$ the two subsets given by
$y = 0$ and $Q(u,v) = 0$. It is straightforward to check that all the surfaces defined by the equations \eqref{equation Munshi}
or \eqref{equation surface} have two singularities of type $\mathbf{D}_4$ over $\overline{\mathbb{Q}}$.

Let us note that, all along this article, the constants involved in the notations $\ll$ and $\gg$ may depend on
$\varepsilon$, $\mathbf{e}$ and $Q$.

The main result of this article is the following.

\begin{theorem}
\label{Theorem upper}
Let $\varepsilon > 0$ be fixed. We have the upper bound
\begin{equation*}
N_{U_{\mathbf{e},Q},H}(B) \ll B^{1 + \varepsilon}.
\end{equation*}
\end{theorem}

As in the works of Munshi, the proof of Theorem \ref{Theorem upper} makes use of the natural elliptic fibration to parametrize the
rational points on $U_{\mathbf{e},Q}$. This leads us to investigate the average number of integral points of small naive height on
quadratic twists of a fixed elliptic curve with full rational $2$-torsion. This is the purpose of section \ref{Integral section}.

It is worth mentioning that the analysis of the parametrization of the rational points given by Munshi in \cite{MR2429642} shows
that it should be easy to adapt lemma \ref{geometry lemma} (see section \ref{Geometry section}) to prove that the surfaces
defined by \eqref{equation Munshi} and considered by Munshi also have linear growth.

Another interesting problem is to prove sharp lower bounds for $N_{U_{\mathbf{e},Q},H}(B)$. A simpler way to state this
is to ask what can be said about the quantity
\begin{equation}
\label{definition beta}
\beta_{U_{\mathbf{e},Q}}(B) = \frac{\log N_{U_{\mathbf{e},Q},H}(B)}{\log(B)}.
\end{equation}
In the following, we choose to take $Q(u,v) = uv$, even though similar results could be proved for other
choices of $Q$. We respectively call $V_{\mathbf{e}}$ and $U_{\mathbf{e}}$ the surface and the open subset
corresponding to this choice. We establish the following result.

\begin{corollary}
\label{Corollary beta}
The limit of $\beta_{U_{\mathbf{e}}}(B)$ as $B$ tends to $+ \infty$ exists and equals $1$. More precisely, we have
\begin{eqnarray*}
\beta_{U_{\mathbf{e}}}(B) = 1 + O\left( \frac1{\log \log B} \right).
\end{eqnarray*}
\end{corollary}

To prove the lower bound $B (\log B)^8 \ll N_{U_{\mathbf{e}},H}(B)$, which is conjecturally best possible, a natural idea is to
make use of universal torsors above $V_{\mathbf{e}}$. Indeed, this strategy has been successful to establish Manin's conjecture
for several examples of singular del Pezzo surfaces of low degree (see \cite{dP2E7} and \cite{D4} for the most striking results).
Hausen and S\"{u}ss \cite[Example $5$.$5$]{MR2671185} have computed the equations of such a torsor and it turns out that proving this
lower bound does not seem to be immediate. It would be interesting to solve this problem.

\subsection{Integral points of small height on quadratic twists}

\label{Integral section}

Given a family of non-rational curves, which is reasonable in some sense, a loose general expectation is that most curves do not
have any integral point. Proving results of this flavour is expected to be very hard, even in some simple cases.
Another expectation is that the smallest (rational point and a fortiori) integral point on the curves having at
least one, is most likely not small. This second statement is easier to approach and the aim of this section is to
investigate the case of a family of quadratic twists of a fixed elliptic curve with full rational $2$-torsion.

For $n \geq 1$, we introduce the elliptic curve $E_{n,\mathbf{e}}$ defined by the equation
\begin{equation*}
y^2 = (x - e_1 n) (x - e_2 n) (x - e_3 n).
\end{equation*}
We instantly check that the curves $E_{1,\mathbf{e}}$ and $E_{n,\mathbf{e}}$ are isomorphic over $\mathbb{Q}(\sqrt{n})$. Our
interest lies in the set of integral points on $E_{n,\mathbf{e}}$, so we set
\begin{equation*}
E_{n,\mathbf{e}}(\mathbb{Z}) = \{ P \in E_{n,\mathbf{e}}(\mathbb{Q}), x(P) \in \mathbb{Z} \},
\end{equation*}
and also
\begin{equation*}
E_{n,\mathbf{e}}^{\ast}(\mathbb{Z}) = \{ P \in E_{n,\mathbf{e}}(\mathbb{Z}), y(P) \neq 0 \}.
\end{equation*}
The elements of $E_{n,\mathbf{e}}^{\ast}(\mathbb{Z})$ will be referred to as the non-trivial integral points on
$E_{n,\mathbf{e}}$.

As already explained, a difficult problem is to obtain upper bounds for the number of $n \leq N$ such
that $E_{n,\mathbf{e}}$ has at least one non-trivial integral point. Following the philosophy described above, it is reasonable
to expect that this set has density $0$ but the proof of this statement seems to be out of reach.

An easier problem is to investigate properties of integral points of bounded height on the curves $E_{n,\mathbf{e}}$ on average
over $n$. Given $P \in E_{n,\mathbf{e}}(\mathbb{Z})$ with coordinates $(x,y) \in \mathbb{Z}^2$, we define its exponential naive height
$\mathcal{H}(P)$ by setting
\begin{equation*}
\mathcal{H}(P) = \max \{ |y|^{1/3}, |x|^{1/2} \}.
\end{equation*}

The following theorem will be the key result in the proof of Theorem \ref{Theorem upper}. It gives lower and upper bounds for
the number of non-trivial integral points of bounded height on the curves $E_{n,\mathbf{e}}$ on average over $n$.

\begin{theorem}
\label{Theorem Integral}
We have the bounds
\begin{equation*}
B \ll \sum_{n \geq 1} \# \{ P \in E_{n,\mathbf{e}}^{\ast}(\mathbb{Z}), \mathcal{H}(P) \leq B \}
\ll B (\log B)^{\delta_{\mathbf{e}}},
\end{equation*}
where $\delta_{\mathbf{e}} = 4$ if $e_1 e_2 e_3 \neq 0$ and $\delta_{\mathbf{e}} = 6$ otherwise.
\end{theorem}

Note that the interest of Theorem \ref{Theorem Integral} mainly lies in the upper bound, and the lower bound implies that it is
sharp up to the factor $(\log B)^{\delta_{\mathbf{e}}}$. Let us note here that with more work, the lower bound could be improved
by a factor $(\log B)^4$ but this would not change anything in our other results.

It is not hard to check that there exists an integer $n \gg B^2$ for which the set
$\{ P \in E_{n,\mathbf{e}}^{\ast}(\mathbb{Z}), \mathcal{H}(P) \leq B \}$ is not empty. Therefore, the upper bound in
Theorem~\ref{Theorem Integral} states that most quadratic twists of $E_{1,\mathbf{e}}$ do not have a non-trivial integral
point of small height.

To be more specific, the upper bound in Theorem \ref{Theorem Integral} allows us to establish the following density statement.

\begin{corollary}
\label{Corollary integral}
Let $A > 6$ be fixed. The set of $n \geq 1$ such that every $P \in E_{n,\mathbf{e}}^{\ast}(\mathbb{Z})$ satisfies
\begin{equation*}
\mathcal{H}(P) > n^{1/2} (\log n)^{-A},
\end{equation*}
has density $1$ in the set of $n \geq 1$ such that $E_{n,\mathbf{e}}^{\ast}(\mathbb{Z}) \neq \emptyset$.
\end{corollary}

It is worth pointing out that it is easy to check that if $e_1 e_2 e_3 \neq 0$ then any $P \in E_{n,\mathbf{e}}(\mathbb{Z})$ has to
satisfy $\mathcal{H}(P) \gg n^{1/2}$, so Corollary~\ref{Corollary integral} actually holds for any $A>0$ in this case.

\subsection{Outline of the article}

We start by establishing Theorem \ref{Theorem Integral}. The proof of this result goes in two steps. The first step consists in
using the fact that $E_{n,\mathbf{e}}$ has full rational $2$-torsion to parametrize the integral points on $E_{n,\mathbf{e}}$
using a complete $2$-descent. This is achieved in section \ref{Descent section}. In the second step, we bound the number of
non-trivial integral points of bounded height on the curves $E_{n,\mathbf{e}}$ on average over~$n$. To achieve this, we appeal
to the recent result of the author \cite[Lemma $4$]{Bihomogeneous}. This lemma is stated in section \ref{Geometry section}.

Corollary \ref{Corollary integral} straightforwardly follows from the upper bound in Theorem \ref{Theorem Integral} after noticing
that the number of $n \leq N$ for which the set $E_{n,\mathbf{e}}^{\ast}(\mathbb{Z})$ is not empty is $\gg N^{1/2}$.

Finally, we prove Theorem \ref{Theorem upper} using the natural elliptic fibration and the upper bound in
Theorem \ref{Theorem Integral}. Corollary \ref{Corollary beta} also follows from this upper bound, together with the lower bound
$B \ll N_{U_{\mathbf{e}},H}(B)$.

\subsection{Acknowledgements}

It is a pleasure for the author to thank R\'egis de la Bret\`eche, Timothy Browning, Daniel Loughran, Dave Mendes da Costa,
Peter Sarnak, Arul Shankar and Anders S\"{o}dergren for interesting and stimulating conversations related to the topics of this
article.

The financial support and the perfect working conditions provided by the Institute for Advanced Study are gratefully acknowledged.
This material is based upon work supported by the National Science Foundation under agreement No. DMS-$1128155$. Any opinions,
findings and conclusions or recommendations expressed in this material are those of the author and do not necessarily reflect
the views of the National Science Foundation.

\section{Preliminaries}

\subsection{Descent argument}

\label{Descent section}

In this section, we derive a convenient parametrization of the integral points on $E_{n,\mathbf{e}}$ using the fact that
$E_{n,\mathbf{e}}$ has full rational $2$-torsion. We start by proving the following elementary lemma.

\begin{lemma}
\label{descent lemma}
Let $(y, x_1, x_2, x_3) \in \mathbb{Z}_{\neq 0}^4$ be such that $y^2 = x_1 x_2 x_3$.
There exists a unique way to write
\begin{equation*}
x_i = d_j d_k w^2 a_i^2 a_j a_k b_i^2,
\end{equation*}
for $\{ i, j, k \} = \{ 1, 2, 3 \}$ and
\begin{equation*}
y = d_1 d_2 d_3 w^3 a_1^2 a_2^2 a_3^2 b_1 b_2 b_3,
\end{equation*}
where $(d_1, d_2, d_3, w, a_1, a_2, a_3, b_1, b_2, b_3) \in \mathbb{Z}_{\neq 0}^4 \times \mathbb{Z}_{> 0}^6$
is subject to the conditions $|\mu(a_i)| = 1$ and $\gcd(d_i a_j b_j, d_j a_i b_i) = 1$ for $i, j \in \{1, 2, 3 \}$, $i \neq j$,
and $d_1 d_2 d_3 > 0$.
\end{lemma}

\begin{proof}
Let us set $x = \gcd(x_1, x_2, x_3)$ and let us write $x_i = x x_i'$ for $i \in \{ 1, 2, 3 \}$, where $\gcd(x_1', x_2', x_3') = 1$.
We see that $x \mid y$ and we can thus write $y = x y'$. We obtain
\begin{equation*}
y'^2 = x x_1' x_2' x_3'.
\end{equation*}
Let us now set $d_i = \sign(x_i') \gcd(x_j', x_k')$ for $\{ i, j, k \} = \{ 1, 2, 3 \}$. Let us note that we have $d_1 d_2 d_3 > 0$.
We can write $x_i' = d_j d_k \xi_i$ with $\xi_i > 0$ for $\{ i, j, k \} = \{ 1, 2, 3 \}$, where $\gcd(d_i \xi_j, d_j \xi_i) = 1$
for $i, j \in \{1, 2, 3 \}$, $i \neq j$. Since $d_1 d_2 d_3 \mid y'$, we can write $y' = d_1 d_2 d_3 z$. We thus get
\begin{equation*}
z^2 = x \xi_1 \xi_2 \xi_3.
\end{equation*}
There is a unique way to write $\xi_i = a_i b_i^2$ with $a_i, b_i > 0$ and $|\mu(a_i)| = 1$ for $i \in \{1, 2, 3 \}$. We see that
$b_1 b_2 b_3 \mid z$ so we can write $z = b_1 b_2 b_3 z'$. We finally obtain
\begin{equation*}
z'^2 = x a_1 a_2 a_3.
\end{equation*}
Since $a_1$, $a_2$ and $a_3$ are squarefree and pairwise coprime, this implies that we can write $x = w^2 a_1 a_2 a_3$ and
$z' = w a_1 a_2 a_3$, which completes the proof.
\end{proof}

Lemma \ref{descent lemma} immediately implies the following result, which provides us with the desired parametrization of the
non-trivial integral points on $E_{n,\mathbf{e}}$.

\begin{lemma}
\label{parametrization}
There is a bijection between the set of non-trivial integral points on $E_{n,\mathbf{e}}$ and the set of
$(d_1, d_2, d_3, w, a_1, a_2, a_3, b_1, b_2, b_3) \in \mathbb{Z}_{\neq 0}^4 \times \mathbb{Z}_{> 0}^6$
satisfying, for $\{i, j, k \} = \{1, 2, 3 \}$, the equations
\begin{equation*}
(e_i - e_j) n = d_k w^2 a_1 a_2 a_3 (d_i a_j b_j^2 - d_j a_i b_i^2),
\end{equation*}
and the conditions $|\mu(a_i)| = 1$ and $\gcd(d_i a_j b_j, d_j a_i b_i) = 1$ for $i, j \in \{1, 2, 3 \}$, $i \neq j$,
and $d_1 d_2 d_3 > 0$. This bijection is given, for $P \in E_{n,\mathbf{e}}^{\ast}(\mathbb{Z})$ with coordinates
$(x,y) \in \mathbb{Z}^2$, by
\begin{align*}
x & = e_i n + d_j d_k w^2 a_i^2 a_j a_k b_i^2, \\
y & = d_1 d_2 d_3 w^3 a_1^2 a_2^2 a_3^2 b_1 b_2 b_3,
\end{align*}
for $\{i, j, k \} = \{1, 2, 3 \}$.
\end{lemma}

\subsection{Geometry of numbers}

\label{Geometry section}

The following lemma follows from the recent work of the author \cite[Lemma~$4$]{Bihomogeneous}. It draws upon both geometry of numbers
and analytic number theory tools, and will be the key result in the proof of Theorem \ref{Theorem Integral}.

\begin{lemma}
\label{geometry lemma}
Let $\mathbf{f} = (f_1,f_2,f_3) \in \mathbb{Z}_{\neq 0}^3$ be a vector satisfying the conditions $\gcd(f_i, f_j) = 1$ for
$i, j \in \{1, 2, 3 \}$, $i \neq j$, and let $U_i, V_i \geq 1$ for $i \in \{ 1, 2, 3 \}$. Let also
$N_{\mathbf{f}} = N_{\mathbf{f}}(U_1,U_2,U_3,V_1,V_2,V_3)$ be the number of vectors $(u_1,u_2,u_3) \in \mathbb{Z}_{\neq 0}^3$ and
$(v_1,v_2,v_3) \in \mathbb{Z}_{\neq 0}^3$ satisfying $|u_i| \leq U_i$, $|v_i| \leq V_i$ for $i \in \{ 1, 2, 3 \}$, and the equation
\begin{equation*}
f_1 u_1 v_1^2 + f_2 u_2 v_2^2 + f_3 u_3 v_3^2 = 0,
\end{equation*}
and such that $\gcd(u_i v_i, u_j v_j) = 1$ for $i, j \in \{1, 2, 3 \}$, $i \neq j$. Let $\varepsilon > 0$ be fixed. We have the bound
\begin{equation*}
N_{\mathbf{f}} \ll_{\mathbf{f}} (U_1 U_2 U_3)^{2/3} (V_1 V_2 V_3)^{1/3} M_{\varepsilon}(U_1, U_2, U_3),
\end{equation*}
where
\begin{equation*}
M_{\varepsilon}(U_1, U_2, U_3) = 1 + \max_{\{ i, j, k \} = \{ 1, 2, 3 \} } (U_i U_j)^{-1/2 + \varepsilon} \log 2U_k.
\end{equation*}
\end{lemma}

\section{Integral points of small height on quadratic twists}

\subsection{Proof of Theorem \ref{Theorem Integral}}

Let us first prove the upper bound in Theorem~\ref{Theorem Integral}. Lemma \ref{parametrization} asserts that
$(y,x) \in \mathbb{Z}_{\neq 0} \times \mathbb{Z}$ satisfies the equation
\begin{equation}
\label{equation curve}
y^2 = (x - e_1 n) (x - e_2 n) (x - e_3 n),
\end{equation}
if and only if $x$ and $y$ can be written, for $\{i, j, k \} = \{1, 2, 3 \}$, as
\begin{align*}
x & = e_i n + d_j d_k w^2 a_i^2 a_j a_k b_i^2, \\
y & = d_1 d_2 d_3 w^3 a_1^2 a_2^2 a_3^2 b_1 b_2 b_3,
\end{align*}
where $(d_1, d_2, d_3, w, a_1, a_2, a_3, b_1, b_2, b_3) \in \mathbb{Z}_{\neq 0}^4 \times \mathbb{Z}_{> 0}^6$ satisfies,
for $\{i, j, k \} = \{1, 2, 3 \}$, the equations
\begin{equation}
\label{equations n}
(e_i - e_j) n = d_k w^2 a_1 a_2 a_3 (d_i a_j b_j^2 - d_j a_i b_i^2),
\end{equation}
and the conditions $|\mu(a_i)| = 1$ and $\gcd(d_i a_j b_j, d_j a_i b_i) = 1$ for $i, j \in \{1, 2, 3 \}$, $i \neq j$,
and $d_1 d_2 d_3 > 0$. The equations \eqref{equations n} can have a solution $n \in \mathbb{Z}_{> 0}$ only if
\begin{equation*}
(e_2 - e_3) d_2 d_3 a_1 b_1^2 + (e_3 - e_1) d_1 d_3 a_2 b_2^2 + (e_1 - e_2) d_1 d_2 a_3 b_3^2 = 0.
\end{equation*}
Moreover, since $e_1$, $e_2$ and $e_3$ are distinct, there is at most one such solution $n \in \mathbb{Z}_{> 0}$.
The conditions $\gcd(d_i, d_j a_i b_i) = 1$ imply that $d_i \mid e_j - e_k$ for $\{i, j, k\} = \{1, 2, 3 \}$ so we can write
$e_2 - e_3 = d_1 c_1$, $e_3 - e_1 = d_2 c_2$ and $e_1 - e_2 = d_3 c_3$. We obtain the equation
\begin{equation*}
c_1 a_1 b_1^2 + c_2 a_2 b_2^2 + c_3 a_3 b_3^2 = 0.
\end{equation*}
Let us call $h = \gcd (c_1, c_2, c_3)$ and let us write $c_i = h f_i$ for $i \in \{1, 2, 3 \}$. We thus have $\gcd(f_1, f_2, f_3) = 1$.
From the two relations $d_1 f_1 + d_2 f_2 + d_3 f_3 = 0$ and
\begin{equation}
\label{equation final}
f_1 a_1 b_1^2 + f_2 a_2 b_2^2 + f_3 a_3 b_3^2 = 0,
\end{equation}
we deduce that $\gcd(f_i, f_j) = 1$ for $i, j \in \{1, 2, 3\}$, $i \neq j$.

From now on, we use the notation $\mathbf{f} = (f_1, f_2, f_3)$. We let $\mathcal{N}_{\mathbf{f}}(B)$ be the number of
$(w, a_1, a_2, a_3, b_1, b_2, b_3) \in \mathbb{Z}_{\neq 0} \times \mathbb{Z}_{> 0}^6$ satisfying the equation \eqref{equation final},
the inequality
\begin{equation*}
|w|^3 a_1^2 a_2^2 a_3^2 b_1 b_2 b_3 \leq B^3,
\end{equation*}
and the conditions $\gcd(a_i b_i, a_j b_j) = 1$ for $i, j \in \{1, 2, 3 \}$, $i \neq j$. The investigation above shows that
\begin{equation*}
\sum_{n \geq 1} \# \{ P \in E_{n,\mathbf{e}}^{\ast}(\mathbb{Z}), \mathcal{H}(P) \leq B \}
\ll \max_{\mathbf{f}} \mathcal{N}_{\mathbf{f}}(B),
\end{equation*}
where the maximum is taken over $\mathbf{f}$ satisfying $f_i \mid e_j - e_k$ for $\{i, j, k\} = \{ 1, 2, 3 \}$ and
$\gcd(f_i, f_j) = 1$ for $i, j \in \{1, 2, 3 \}$, $i \neq j$.

We have thus proved that it is sufficient for our purpose to bound the quantity $\mathcal{N}_{\mathbf{f}}(B)$. To achieve this, for
$i \in \{ 1, 2, 3 \}$, we let $W, A_i, B_i \geq 1$ run over the set of powers of $2$ and we define
$\mathcal{M}_{\mathbf{f}} = \mathcal{M}_{\mathbf{f}}(W,A_1,A_2,A_3,B_1,B_2,B_3)$ as the number of
$(w, a_1, a_2, a_3, b_1, b_2, b_3) \in \mathbb{Z}_{\neq 0} \times \mathbb{Z}_{> 0}^6$ satisfying the equation \eqref{equation final},
the conditions $W < |w| \leq 2 W$, $A_i < a_i \leq 2 A_i$ and $B_i < b_i \leq 2 B_i$, and $\gcd(a_i b_i, a_j b_j) = 1$ for
$i, j \in \{1, 2, 3 \}$, $i \neq j$. We have
\begin{equation*}
\mathcal{N}_{\mathbf{f}}(B) \ll \sum_{\substack{W, A_i, B_i \\ i \in \{ 1, 2, 3 \}}} \mathcal{M}_{\mathbf{f}},
\end{equation*}
where the sum is over $W, A_i, B_i \geq 1$, $i \in \{ 1, 2, 3 \}$, satisfying the inequality
\begin{equation}
\label{height 1}
W^3 A_1^2 A_2^2 A_3^2 B_1 B_2 B_3 \leq B^3.
\end{equation}
Lemma \ref{geometry lemma} gives the upper bound
\begin{equation*}
\mathcal{M}_{\mathbf{f}} \ll W (A_1 A_2 A_3)^{2/3} (B_1 B_2 B_3)^{1/3} M_{\varepsilon}(A_1, A_2, A_3),
\end{equation*}
where $M_{\varepsilon}(A_1, A_2, A_3)$ is defined in lemma \ref{geometry lemma}. Choosing for instance $\varepsilon = 1/4$ and
summing over $W$ using the condition \eqref{height 1}, we finally obtain
\begin{align*}
\mathcal{N}_{\mathbf{f}}(B) & \ll \sum_{\substack{W, A_i, B_i \\ i \in \{ 1, 2, 3 \}}}
W (A_1 A_2 A_3)^{2/3} (B_1 B_2 B_3)^{1/3} M_{1/4}(A_1, A_2, A_3) \\
& \ll B \sum_{\substack{A_i, B_i \\ i \in \{ 1, 2, 3 \}}} M_{1/4}(A_1, A_2, A_3) \\
& \ll B (\log B)^6,
\end{align*}
which completes the first part of the proof of the upper bound in Theorem \ref{Theorem Integral}.

Now let us assume that $e_1 e_2 e_3 \neq 0$ and let us prove that we can take $\delta_{\mathbf{e}} = 4$ in
Theorem \ref{Theorem Integral}. If $n > 2 B^2$ then, since $x = e_i n + d_j d_k w^2 a_i^2 a_j a_k b_i^2$ for
$\{i, j, k\} = \{1, 2, 3 \}$, $|x| \leq B^2$ and $e_1 e_2 e_3 \neq 0$, we have $|d_j d_k| w^2 a_i^2 a_j a_k b_i^2 > B^2$ for
$\{i, j, k\} = \{1, 2, 3 \}$, but this is in contradiction with $|y| \leq B^3$. This implies that
$\{ P \in E_{n,\mathbf{e}}^{\ast}(\mathbb{Z}), \mathcal{H}(P) \leq B \}$ is empty provided that $n > 2 B^2$ so we can
assume that $n \leq 2 B^2$. Therefore, for $\{i, j, k\} = \{1, 2, 3 \}$, we get the conditions
\begin{equation}
\label{3 conditions}
w^2 a_i^2 a_j a_k b_i^2 \leq 3 B^2.
\end{equation}
We now proceed similarly as in the first case. We let $\mathcal{N}_{\mathbf{f}}'(B)$ be the number of
$(w, a_1, a_2, a_3, b_1, b_2, b_3) \in \mathbb{Z}_{\neq 0} \times \mathbb{Z}_{> 0}^6$ satisfying the equation \eqref{equation final},
the inequalities \eqref{3 conditions} and the conditions $\gcd(a_i b_i, a_j b_j) = 1$ for $i, j \in \{1, 2, 3 \}$, $i \neq j$.
Once again, it is sufficient for our purpose to bound $\mathcal{N}_{\mathbf{f}}'(B)$, and we have
\begin{equation*}
\mathcal{N}_{\mathbf{f}}'(B) \ll \sum_{\substack{W, A_i, B_i \\ i \in \{ 1, 2, 3 \}}}
W (A_1 A_2 A_3)^{2/3} (B_1 B_2 B_3)^{1/3} M_{1/4}(A_1, A_2, A_3),
\end{equation*}
where the sum is over $W, A_i, B_i \geq 1$, $i \in \{ 1, 2, 3 \}$, running over the set of powers of $2$ and satisfying,
for $\{i, j, k\} = \{1, 2, 3 \}$, the inequalities
\begin{equation}
\label{height 2}
W^2 A_i^2 A_j A_k B_i^2 \leq 3 B^2.
\end{equation}
Summing over $B_i$, $i \in \{1, 2, 3 \}$, using the conditions \eqref{height 2}, we get
\begin{align*}
\mathcal{N}_{\mathbf{f}}'(B) & \ll B \sum_{\substack{W, A_i \\ i \in \{ 1, 2, 3 \}}} M_{1/4}(A_1, A_2, A_3) \\
& \ll B (\log B)^4,
\end{align*}
as claimed.

Let us now prove the lower bound in Theorem \ref{Theorem Integral}. We can assume by symmetry that
$e_3 > \max \{ e_1, e_2 \}$. Let us denote by $\mathcal{P}_{\mathbf{e}}(B)$ the number of $w \in \mathbb{Z}_{> 0}$ such that
\begin{align*}
4 |(e_1 - e_2)(e_2 - e_3)(e_3 - e_1)| w^3 & \leq B^3, \\
2 |-2 e_1 e_2 + e_1 e_3 + e_2 e_3| w^2 & \leq B^2.
\end{align*}
We remark that $\mathcal{P}_{\mathbf{e}}(B)$ counts one non-trivial integral point on the curves $E_{n,\mathbf{e}}^{\ast}$ for
which $n$ can be written as $n = 2 (2e_3 - e_1 - e_2) w^2$. Note that $2e_3 - e_1 - e_2 > 0$ since
$e_3 > \max \{ e_1, e_2 \}$. Therefore, we have
\begin{equation*}
\sum_{n \geq 1} \# \{ P \in E_{n,\mathbf{e}}^{\ast}(\mathbb{Z}), \mathcal{H}(P) \leq B \} \geq \mathcal{P}_{\mathbf{e}}(B).
\end{equation*}
It is obvious that $\mathcal{P}_{\mathbf{e}}(B) \gg B$, which completes the proof of Theorem \ref{Theorem Integral}.

\subsection{Proof of Corollary \ref{Corollary integral}}

We start by proving the following lemma, which gives a lower bound for the number of $n \leq N$ such that the curve
$E_{n,\mathbf{e}}$ has at least one non-trivial integral point.

\begin{lemma}
\label{lemma lower bound}
We have the lower bound
\begin{equation*}
\# \{ n \leq N, E_{n,\mathbf{e}}^{\ast}(\mathbb{Z}) \neq \emptyset \} \gg N^{1/2}.
\end{equation*}
\end{lemma}

\begin{proof}
As in the proof of the lower bound in Theorem \ref{Theorem Integral}, we can assume that $e_3 > \max \{ e_1, e_2 \}$ and
thus, if $n$ can be written as $n = 2 (2e_3 - e_1 - e_2) w^2$ for some $w \in \mathbb{Z}_{> 0}$, then the equalities
\begin{align*}
y & = 4 (e_1 - e_2) (e_2 - e_3) (e_3 - e_1) w^3, \\
x & = 2 (-2 e_1 e_2 + e_1 e_3 + e_2 e_3) w^2,
\end{align*}
define a non-trivial integral point on $E_{n,\mathbf{e}}$. Noticing that there are $\gg N^{1/2}$ such integers $n \leq N$
completes the proof of the lemma.
\end{proof}

It seems likely that the lower bound in lemma \ref{lemma lower bound} could be improved by a few
$\log N$ factors, but since this slight improvement would not essentially change the statement of Corollary \ref{Corollary integral},
we have decided not to explore this any further.

Let $A > 6$ be fixed. Let $\mathcal{N}_A(N)$ be the number of $n \leq N$ such that there exists
$P \in E_{n,\mathbf{e}}^{\ast}(\mathbb{Z})$ satisfying
\begin{equation*}
\mathcal{H}(P) \leq n^{1/2} (\log n)^{-A}.
\end{equation*}
We have
\begin{equation*}
\mathcal{N}_A(N) \leq
\sum_{n \leq N} \# \{ P \in E_{n,\mathbf{e}}^{\ast}(\mathbb{Z}), \mathcal{H}(P) \leq N^{1/2} (\log N)^{-A} \}.
\end{equation*}
The upper bound in Theorem \ref{Theorem Integral} implies
\begin{equation*}
\mathcal{N}_A(N) \ll N^{1/2} (\log N)^{-A+6},
\end{equation*}
so lemma \ref{lemma lower bound} shows that
$\mathcal{N}_A(N) = o( \# \{ n \leq N, E_{n,\mathbf{e}}^{\ast}(\mathbb{Z}) \neq \emptyset \} )$, which concludes the proof of
Corollary~\ref{Corollary integral}.

\section{Rational points on elliptic fibrations}

\subsection{Proof of Theorem \ref{Theorem upper}}

Recall that $V_{\mathbf{e},Q} \subset \mathbb{P}(3,2,1,1)$ is defined by the equation
\begin{equation}
\label{equation surface 2}
y^2 = (x - e_1 Q(u,v)) (x - e_2 Q(u,v)) (x - e_3 Q(u,v)).
\end{equation}
Thus, we have
\begin{align*}
N_{U_{\mathbf{e},Q},H}(B) & \ll \sum_{\substack{|u|, |v| \leq B \\ Q(u,v) \neq 0}}
\# \{ (y,x) \in \mathbb{Z}_{\neq 0} \times \mathbb{Z}, |y| \leq B^3, |x| \leq B^2, \eqref{equation surface 2} \} \\
& \ll \sum_{n \in \mathbb{Z}_{\neq 0}}
\# \{ (y,x) \in \mathbb{Z}_{\neq 0} \times \mathbb{Z}, |y| \leq B^3, |x| \leq B^2, \eqref{equation curve} \}
\sum_{\substack{|u|, |v| \leq B \\ Q(u,v) = n}} 1.
\end{align*}
Since $Q$ is non-degenerate, we have
\begin{equation*}
\# \{ (u,v) \in \mathbb{Z}^2, |u|, |v| \leq B, Q(u,v) = n \} \ll B^{\varepsilon}.
\end{equation*}
As a result, we get
\begin{equation*}
N_{U_{\mathbf{e},Q},H}(B) \ll B^{\varepsilon}
\sum_{n \in \mathbb{Z}_{\neq 0}} \# \{ P \in E_{n,\mathbf{e}}^{\ast}(\mathbb{Z}), \mathcal{H}(P) \leq B \}.
\end{equation*}
We note that the sum in the right-hand side can be rewritten as
\begin{equation*}
\sum_{n \geq 1} \# \{ P \in E_{n,\mathbf{e}}^{\ast}(\mathbb{Z}), \mathcal{H}(P) \leq B \}
+ \sum_{n \geq 1} \# \{ P \in E_{n,-\mathbf{e}}^{\ast}(\mathbb{Z}), \mathcal{H}(P) \leq B \}.
\end{equation*}
Therefore, using twice the upper bound in Theorem \ref{Theorem Integral}, we obtain
\begin{equation*}
N_{U_{\mathbf{e},Q},H}(B) \ll B^{1 + \varepsilon},
\end{equation*}
which ends the proof of Theorem \ref{Theorem upper}.

\subsection{Proof of Corollary \ref{Corollary beta}}

We proceed exactly as in the proof of Theorem \ref{Theorem upper}. We have
\begin{equation*}
N_{U_{\mathbf{e}},H}(B) \ll \sum_{n \in \mathbb{Z}_{\neq 0}}
\# \{ (y,x) \in \mathbb{Z}_{\neq 0} \times \mathbb{Z}, |y| \leq B^3, |x| \leq B^2, \eqref{equation curve} \}
\sum_{\substack{|u|, |v| \leq B \\ u v = n}} 1.
\end{equation*}
Then, if $n \leq B^2$, we have
\begin{align*}
\# \{ (u,v) \in \mathbb{Z}^2, |u|, |v| \leq B, u v = n \} & \leq 2 \tau(n) \\
& \ll n^{1/ \log \log n} \\
& \ll B^{2 / \log \log B},
\end{align*}
and this upper bound also holds if $n > B^2$. This shows that
\begin{equation*}
N_{U_{\mathbf{e}},H}(B) \ll B^{2 / \log \log B} \sum_{n \in \mathbb{Z}_{\neq 0}}
\# \{ P \in E_{n,\mathbf{e}}^{\ast}(\mathbb{Z}), \mathcal{H}(P) \leq B \}.
\end{equation*}
As in the proof of Theorem \ref{Theorem upper}, using twice the upper bound in Theorem \ref{Theorem Integral}, we obtain
\begin{equation}
\label{upper}
N_{U_{\mathbf{e}},H}(B) \ll B^{1 + 3 / \log \log B}.
\end{equation}

Now let us prove a lower bound for $N_{U_{\mathbf{e}},H}(B)$. Let us assume by symmetry that $e_3 > \max \{ e_1, e_2 \}$ so that
$2e_3 - e_1 - e_2 > 0$, and let us denote by $v_2(m)$ the $2$-adic valuation of an integer $m \geq 1$.
Let $\mathcal{R}_{\mathbf{e}}(B)$ be the number of $(y, x, u, v) \in \mathbb{Z}_{\neq 0}^4$ such that
$\max \{ |y|^{1/3}, |x|^{1/2}, |u|, |v| \} \leq B$ and which can be written as
\begin{align*}
y & = 4 (e_1 - e_2) (e_2 - e_3) (e_3 - e_1) w_1^3 w_2^3, \\
x & = 2 (-2 e_1 e_2 + e_1 e_3 + e_2 e_3) w_1^2 w_2^2, \\
u & = 2^{1 + v_2(2 e_3 - e_1 - e_2)} w_1^2, \\
v & = (2 e_3 - e_1 - e_2) 2^{- v_2(2 e_3 - e_1 - e_2)} w_2^2,
\end{align*}
where $(w_1, w_2) \in \mathbb{Z}^2_{>0}$ satisfies $\gcd(w_1, (2e_3 - e_1 - e_2) w_2) = \gcd(w_2, 2) = 1$. Since $\gcd(u,v)=1$ in the
parametrization above, it is immediate to check that
\begin{equation*}
N_{U_{\mathbf{e}},H}(B) \geq \mathcal{R}_{\mathbf{e}}(B).
\end{equation*}
Since we clearly have $\mathcal{R}_{\mathbf{e}}(B) \gg B$, we have obtained the lower bound
\begin{equation}
\label{lower}
B \ll N_{U_{\mathbf{e}},H}(B).
\end{equation}
Let us note that improving this lower bound by a few $\log B$ factors would not be hard. However, as already explained in the
introduction, proving the lower bound of the expected order of magnitude for $N_{U_{\mathbf{e}},H}(B)$ does not seem to be immediate.

Recalling the definition \eqref{definition beta} of $\beta_{U_{\mathbf{e}}}(B)$, we see that the two bounds \eqref{upper}
and \eqref{lower} complete the proof of Corollary \ref{Corollary beta}.

\bibliographystyle{amsalpha}
\bibliography{biblio}

\end{document}